\documentclass[letterpaper]{article}
\usepackage{amsgen, amsmath, amsfonts, amsthm, amssymb, enumerate,amscd}
\usepackage[all]{xy}
\input xy
\xyoption{all}
\newtheorem{defn}{Definition}[section]
\newtheorem{prop}[defn]{Proposition}
\newtheorem{lem}[defn]{Lemma}
\newtheorem{thm}[defn]{Theorem}
\newtheorem{cor}[defn]{Corollary}
\newtheorem{rem}[defn]{Remark}

\newcommand {\ZZ}{{\mathbb Z}}

\newcommand {\X}{{\bar {X}}}

\newcommand {\C}{{\mathbb C}}

\newcommand {\A}{{\mathcal A}}

\newcommand {\Q}{{\mathbb Q}}
\newcommand {\R}{{\mathbb R}}

\newcommand {\HH}{{\mathfrak  H}}

\newcommand {\m}{{\mathfrak m}}

\newcommand {\CP}{{\mathbb P}}

\newcommand {\I}{{I}}

\newcommand {\ca}{{\mathbf a}}

\newcommand {\AD}{{\mathcal A}^{\bullet}}
\newcommand {\ED}{E^{\bullet}}

\newcommand {\w}{{\omega}}

\newcommand {\MZ}{{\mathcal Z}}

\def\deg{\operatorname{deg}}

\def\Hom{\operatorname{Hom}}
\title{Higher order modular forms and mixed Hodge theory}
\author{Ramesh Sreekantan (Tata Institute of Fundamental Research)}

\begin{document}
\baselineskip=17pt  
\maketitle
\begin{abstract}
In this paper we introduce a certain space of higher order modular
forms of weight $0$ and show that it has a Hodge structure coming
from the geometry of the fundamental group of a modular curve. This
generalises the usual structure on classical weight $2$ forms coming
from the cohomology of the modular curve. Further we construct some
higher order Poincar\'{e} series to get higher order higher weight
forms and using them we define a space of higher weight, higher
order forms which has a mixed Hodge structure as well.
\end{abstract}

\noindent {\bf Mathematics Subject ClassiÞcation (2000) : }11F11, 14C30, 30F30, 32S35. 


\vspace{1cm}

In the theory of automorphic forms, the classical holomorphic
modular forms and their complex conjugates play a special role -
they are the forms most closely linked to geometry. Spaces of such
modular forms can be identified with the cohomology groups of
certain locally constant sheaves on modular curves. From this point
of view certain things, like Hecke operators, become very natural
and this is the first step in associating a motive to a modular form.

Second order modular forms were introduced by Goldfeld and have been
studied in some detail by several people. Examples of such forms we
also discovered in a different context by Kleban and Zagier
\cite{klza}. In \cite{chdios}, the notion of a higher order
automorphic form was considered and the spaces of such forms were
studied. One can speculate as to whether there is any geometry
underlying these spaces of automorphic forms. In general, these spaces
can be rather large and one cannot expect much. However, in this
paper we consider certain subspaces with additional structure that
can be understood as the generalization of the space of classical
holomorphic modular forms. We show that these spaces can be
identified with certain spaces coming from the fundamental groups of
modular curves.

In weight $0$, the spaces we consider are the spaces of
anti-derivatives of iterated integrals of smooth 1-forms. In the
classical situation the Eichler-Shimura isomorphism identifies the
spaces of holomorphic modular forms along with the conjugates of cusp
forms with the cohomology groups of locally constant sheaves on a modular curve and
these spaces have mixed Hodge structures. We show that our  spaces
can be identified with graded quotients of the dual of the group
ring of the fundamental group of the modular curve generalizing the
Eichler-Shimura theorem. These quotients have mixed Hodge structures
due to Hain \cite{hain} and Morgan. Using this, we can define a
mixed Hodge structure on the space of weight $0$ higher order forms.
In general, this Hodge structure depends on the base point. Deligne and Goncharov \cite{dego}
show that this Hodge structure is the same as that on a cohomology group of a pair.

For higher weights, we generalize a construction of Goldfeld and
O'Sullivan \cite{goos} using twisted Poincar\'{e} series to construct 
some higher order higher weight forms. We show
that, when the weight $k>2$, these spaces of such forms also have a mixed Hodge
structure.

One purpose of relating the higher order forms to the geometry of
the modular curve is to define a Hecke theory. Classically the Hecke
operators are the operators induced on the cohomology groups by
certain algebraic correspondences. In the higher order case,
however, one cannot use this as Hecke correspondences do not induce
maps between homotopy groups or on the corresponding Hodge
structures. This perhaps explains why there is no satisfactory Hecke
theory for higher order forms.

In weight $0$, much of the theory is analogous to the theory of
Multiple Zeta values and the geometry of
$\mathbb{P}^1-\{0,1,\infty\}$ due to Deligne, Goncharov \cite{dego}
and others. From this point of view the higher order modular forms
can be viewed as generalizations of the single variable multiple
polylogarithms.

{\em Acknowledgements} I would like to thank Nikolaos Diamantis for
careful reading and comments on several earlier versions of this
manuscript. I would also like to thank G.V. Ravindra, J.G. Biswas
and D. Pancholi  and N.Fakhruddin for their comments and help. I would like to thank the
Max-Planck Institute in Bonn and the TIFR Centre in Bangalore for
their hospitality when this work was done. Finally I would like to thank the referee
for suggesting a considerably simpler argument for the main theorem.

\section{Higher order automorphic  forms}

Let $\Gamma$ be a discrete subgroup of $SL_2(\R)$ with no elliptic
fixed points, so $\Gamma$ is isomorphic to the fundamental group of
$X=\Gamma \backslash \HH$, $\pi_1(X, x_0)$ at some point $x_0 \in X$. Let
$\ZZ[\Gamma]$ be the group ring and $J=J_{x_0}=<\gamma-1>$ the augmentation
ideal of $\ZZ[\Gamma]$ which fits in the exact sequence.
$$0 \rightarrow J \rightarrow \ZZ[\Gamma] \stackrel{\deg}{\rightarrow}
\ZZ \rightarrow 0$$
If $f:\HH \rightarrow \C$ is a function  and $k \in \ZZ$ define
$$(f|_k{\gamma})(z)=j(\gamma,z)^{-k}f(\gamma z)$$
where $\gamma=\begin{pmatrix} a & b\\c &d \end{pmatrix}$ is in
$SL_2(\R)$ and $j(\gamma,z)=(cz+d)$ is the usual automorphy factor.
We extend this to an action of $\ZZ[\Gamma]$ on the space of
functions by defining
$$f|_k{\sum a_i \gamma_i}=\sum a_i f|_{k}\gamma_i$$
An {\em automorphic form of weight $k $ for $\Gamma$} is a function
such that
$$f|_k \gamma=0$$
for all $\gamma$ in $J$. More generally, for $s \in {\mathbb N}$,  we define a {\em higher
order automorphic form of weight $k$ and order $s$ for $\Gamma$} to
be a function $f:\HH \rightarrow \C$ such that
$$f|_k \gamma=0$$
for all $\gamma$ in $J^s$. Let $M_k^s=M_k^s(\Gamma)$ denote the
space of higher order automorphic forms of order $s$. We further define $M_k^0(\Gamma)$ to be the constants $\C$.
For $s=1$ with the added conditions of holomorphy and growth at the cusps,  gives  the classical modular forms of weight $k$. 
For a fixed weight $k$ there is a natural filtration by order
$$M_k^0 \subseteq M_k^1 \subseteq M_k^2 \dots M_k^{s} \subseteq M_k^{s+1} \dots$$
These spaces can be rather large, so one cannot expect much
structure. In this paper we study a certain subspaces of the space
of these forms which have a lot more structure. 

\section{Iterated Integrals}

Let $X$ be a smooth manifold with a point $x_0$.  Let $P(X)=P(X, x_0)$
denote the space of loops on $X$ based at $x_0$, namely, continuous functions 
$$\gamma:[0,1] \rightarrow X, \gamma(0)=\gamma(1)=x_0.$$
A function $\phi:P(X) \rightarrow R$, where $R$ is a ring, is said to be a {\em homotopy
functional} if $\phi$ depends only on the homotopy class of
$\gamma$. That is,  it defines a function on $\Gamma=\pi_1(X, x_0)$
or equivalently an element of $\Hom_{Ab Grps}(\ZZ[\Gamma],R)$

Let $\ED(X)$ denote the de Rham complex of smooth forms on $X$. It is a differential graded algebra (dga) -- where a differential graded algebra is a graded algebra $A$ with a degree $1$ map $d:A \longrightarrow A$ such that 
\begin{itemize}
\item $d \circ d=0$
\item $d(a \cdot b) = d(a)\cdot b +(-1)^{\deg(a)} a \cdot d(b)$
\end{itemize}
Let $\w$ be a 1-form in  $\A^1(X)$, where $\AD(X)$ is a sub dga of
$\ED(X)$. The map
$$\gamma \rightarrow \int_{\gamma} \w=\int_{0}^{1} f(t)dt$$
where  $\gamma^*(\w)=f(t)dt$, defines  a function  on $P(X)$.
This defines an element of $\Hom(\ZZ[\Gamma],\R)$ if and only if
$\w$ is closed. Hence this only detects elements of $\Gamma$ visible
in the homology of $X$. It vanishes on $J^2$ as if $(\alpha-1)(\beta-1) \in J^2$, then 
$$\int_{\alpha\beta} \w=\int_{\alpha} \w + \int_{\beta} w$$
so
$$\int_{\alpha\beta} \w -\int_{\alpha} \w -\int_{\beta} \w + \int_{1} \w=0.$$

The iterated integrals studied by Chen \cite{chen} detect more
elements of the group ring.  Suppose $\w_1,\w_2,...,\w_r$ are smooth
1-forms in $\A^1(X)$ and $\gamma$ is a path on $X$. Define
\begin{equation}
\int_{\gamma} \w_1\w_2...\w_r= {\int ...\int}_{0\leq t_1 \leq t_2
...\leq t_r \leq 1} f_1(t_1)f_2(t_2)...f_r(t_r) dt_1dt_2...dt_r
\end{equation}
where $\gamma^{*}(\w_i)=f_i(t)dt$. This defines  a function on the space of paths of $X$  which will be
denoted by  $\int \w_1...\w_r$ and  is called an {\em iterated line
integral of  length $r$}.   A linear combination  of such functions
is called an {\em iterated integral} and  its length is the length
of the longest line integral. Length $0$ iterated integrals are defined to be constant functions. Let $B_s(\AD(X))$ denote the space of iterated integrals of length $\leq s$ coming from forms in $\AD$.

An iterated integral is not necessarily invariant under  homotopy. Chen\cite{chen} formulated a  condition in terms of differential graded algebras under which iterated integrals which are homotopy functionals are closed with respect to a certain differential. However, we have no use for that formalism in what follows so we will not describe it. It does underlie the following notation though. Let $H^0(B_s(\AD(X)), x_0)$ be the space of iterated integrals of length $\leq s$ which are homotopy functionals on loops based at a
point $x_0$ on $X$ modulo those iterated integrals which integrate to $0$ along any path. We shorten this to $H^0(B_s(X), x_0)$ if $\AD(X)$ is $\ED(X)$. 

If $I$ is in $B_s(\AD(X))$ and $\gamma$ in $\Gamma$. Let
$$\displaystyle{<I,\gamma>=\int_{\gamma} I}$$
denote the evaluation map. This can be extended by linearity to all of $\ZZ[\Gamma]$. Let $H^0(\bar{B}_s(\AD(X), x_0)$ denote the subspace of $H^0(B_s(\AD(X)), x_0)$ such that
$$<I,\eta_{x_0}>=0$$
where $\eta_{x_0}$ denotes the constant loop at $x_0$. Namely, these are iterated integrals with constant term being $0$. 

We have the following propositions that can be found in \cite{hain}.
As we will have to appeal to them several times we find it useful to repeat them  here.

\begin{prop}[Hain\cite{hain}, Proposition 2.9] Let $\int \w_1\dots\w_s$ be an iterated line integral and $\alpha$
and $\beta$ two paths such that $\alpha(1)=\beta(0)$. Then
$$\int_{\alpha\beta} \w_1\dots \w_s=\sum_{i=0}^{s} \int_{\alpha} \w_1\dots
\w_i \int_{\beta} \w_{i+1} \dots \w_s$$
where an empty integral is to be understood as 1. 

\label{alphabeta}
\end{prop}
As a corollary one has,
\begin{cor}
 If $\alpha_i$ are loops and $\beta(0)=\alpha_i(1)$, then
$$\int_{\prod_k (\alpha_k-1)\beta} \w_1\dots \w_s=\sum_{i=1}^{s} \int_{\prod_k (\alpha_k-1)} \w_1\dots \w_{i} \int_{\beta} \w_{i+1}\dots \w_s$$
\label{alphabetacor}
\end{cor}
Let 
$$p_s:H^0(B_s(X), x_0) \longrightarrow \bigotimes^s H^1(X,\C)$$
be defined as follows. For $\alpha_1 \otimes \dots \otimes \alpha_s \in \otimes H_1(X,\C)$, where $\alpha_i$ are loops based at $x_0$, 
$$p_s(I)( \otimes_{i=1}^{s} \alpha_i)=\left < I,\prod(\alpha_i-1) \right >$$

\begin{prop}[Hain\cite{hain}, Prop 2.10, Prop 2.13] 

If $\w_1 \dots \w_r$ are smooth one-forms on $X$ and $\alpha_1,\dots \alpha_s$ are loops
based at $x_0$ then
\begin{equation}
\left<\int \w_1\dots \w_r,\prod_{i=1}^{s} (\alpha_i-1)
\right>=\begin{cases} \prod_{i=1}^{s} \int_{\alpha_j} \w_j \text
{ if  r=s}\\
0 \text{ if r $<$ s}. \end{cases} \label{period}
\end{equation}
\label{hainproduct}
\end{prop}

Finally we state another proposition which allows us to reduce the length of an iterated integral if one of the terms is exact. 

\begin{prop}[Hain\cite{hain}, Prop 1.3]
\label{exact}
Suppose $\w_1,\dots \w_s$ are $1$-forms on $X$ and $\gamma$ is a path on $X$.  If $f$ is a function on $M$ then 
\begin{itemize}

\item $$\int_{\gamma} df  \w_1\dots \w_s=\int_{\gamma} (f\w_1)\w_2\dots \w_s - f (\gamma(0))\int_{\gamma} \w_1\dots \w_s.$$

\item $$\int_{\gamma} \w_1\dots \w_{i-1} df  \w_{i}\dots \w_s=\int_{\gamma} \w_1\dots (f\w_{i})\w_{i+1}\dots \w_s$$ 
$$- \int_{\gamma} \w_1\dots (f\w_{i-1})\w_i\dots \w_s$$

\item $$\int_{\gamma} \w_1\dots \w_s df = f(\gamma(1))\int_{\gamma} \w_1\dots \w_s - \int_{\gamma} \w_1\dots \w_{s-1} (f \w_s)$$

\end{itemize}

\end{prop}
 This shows there exist non-trivial iterated integrals which integrate to $0$ along any path.

\section{Higher order modular forms of Geometric origin}

Now let $X=\Gamma \backslash \HH$ where $\Gamma$ is an arithmetic subgroup of
$SL_2(\R)$ and let $\pi:\HH \rightarrow X$ be the canonical map. Let $x_0$ be a point on $X$. We
further assume that $\Gamma$ has no elliptic fixed points, so
$\Gamma \simeq \pi_1(X, x_0)$. In this section we define a  subspace
of the space of higher order automorphic forms of weight $0$ for
$\Gamma$ which has an additional structure coming from the 
geometry of the curve $X$.  As $X$ is a curve, it has  the cuspidal compactification 
$\bar{X}=\Gamma \backslash (\HH \cup \CP^1(\Q))$, which is a smooth projective curve. We also define a subspace of this space coming  from $\bar{X}$ 
which generalize the classical cusp forms. Let $D=\Gamma \backslash \CP^1(\Q)$ be the set of cusps -- $\bar{X}=X \cup D$.

We have the following theorem

\begin{thm} Let $I$ be  an iterated integral on $X$  of length $\leq s$ which is a homotopy functional.  Let $x_0$ denote a point on $X$. Let $z_0$ denote a point of $\HH$ lying in the fibre over $x_0$. Then the function $F_{I}$  on $\HH$ defined by
$$F_{I}(z)=\int_{z_0}^z \pi^{*}(I):=\int_{z_0}^{z} I$$
is a higher order modular form of order $s+1$. Here and from now on by abuse of notation we use $I$ to denote both the integral on $X$ as well as its pullback $\pi^*(I)$ on $\HH$.  As before we define an empty integral to be $1$ and hence integrals of iterated integrals of length $0$, namely constants, are just constants. 

Further, this gives a well defined injective linear map
$$H^0(B_s(X), x_0) \stackrel{\Psi}{\longrightarrow} M_0^{s+1}(\Gamma) $$
\label{mainthm}
\end{thm}

\begin{proof} Let $I$ be  homotopy functional of length $\leq s$. This has an expression of the form 
$$I=\sum_{|J| \leq s} \w_J $$
where $\w_J=\w_{j_1}\dots \w_{j_{r_J}}$ are iterated line integrals.  We need to show that for any 
$\gamma_1,\gamma_2,\dots,\gamma_{s+1} \in \Gamma$,
$$ F_{I}|_{\prod_{i=1}^{s+1} (\gamma_i-1)}(z)=0.$$
Let $\eta$ be a path from $z_0$ to $z$ on $\HH$. If $\gamma$ is a loop on $X$ based at $x_0$, then one has a composite path $\gamma \pi(\eta)$ on $X$ from $x_0$ to $\pi(z)$.  This can be lifted to a unique path on $\HH$ from $z_0$ to $\gamma z$ passing through  $\gamma z_0$.  We denote it by $\gamma\eta$. 

Notice that
$$F_{I}|_{\gamma_1-1}(z)=\int_{z_0}^{\gamma_1 z} I
-\int_{z_0}^{z} I=\int_{\gamma_1\eta} I- \int_{\eta} I=\int_{(\gamma_1-1)\eta} I$$
as 
$$\int_{\gamma z_0}^{\gamma z} I= \int_{z_0}^{z} I$$
since $I$ is $\Gamma$-invariant. 

For each iterated line integral  $\w_J=\w_{j_1}\dots \w_{j_r}$  appearing in $I$ we can apply Corollary \ref{alphabetacor} and Proposition \ref{hainproduct} to get 
$$\int_{\prod_{i=1}^{s}  (\gamma_i-1)\eta}  \w_J=\sum_{i=1}^{s} \int_{\prod_{k=1}^{s}(\gamma_i-1)} \w_{j_1}\dots \w_{j_{i}} \int_{\eta} \w_{j_{i+1}} \dots \w_{j_s}=
\begin{cases} \prod_{i=1}^{s} \int_{\gamma_i} \w_{j_i} \text
{ if  r=s}\\
0 \text{ if r $<$ s}. \end{cases} 
$$
Therefore 
$$F_I|_{\prod_{i=1}^{s} (\gamma_i-1)}(z)=\sum_{|J|=s} \prod_{i=1}^{s} \int_{\gamma_i} \w_{j_i}$$
and this expression is independent of $z$. In particular, applying $\gamma_{s+1}-1$ annihilates it so $F_I|_{J^{s+1}} \equiv 0$. 

It should be remarked that since $I$ is a homotopy functional, $F_I(z)$ does not depend on the path from $z_0$ to $z$ but the iterated line integrals  $\w_J$ need not be homotopy invariant and hence we had to choose paths. However, by construction, the sums of iterated line integrals  that appear in the expression of $F_I|_{\gamma_{1} - 1)}(z)$ as a sum of products of iterated line integrals are homotopy functionals. That is, we can collect terms together such that 
$$
\int_{z_0}^{\gamma_1 z} I
-\int_{z_0}^{z} I=\int_{(\gamma_1-1)\eta} I=\sum_{r=1}^{s} \int_{z_0}^{\gamma_1z_0} I_1^{r} \int_{z_0}^{z} I_2^{r}  $$
where  $I_1^r$ and $I_2^r$ are homotopy functionals. Note that $I_2^s=I_1^0=I$. 

To prove injectivity, suppose $I$ is of order $s$ and  $F_{I} \equiv 0$. Then in particular, for any $\gamma$ in $\Gamma$, 
$$F_I(\gamma z_0)=\int_{z_0}^{\gamma z_0} I=0$$
hence $I$ is $0$ as a homotopy functional. Hence $I=0$ in $H^0(B_s(X), x_0)$.

\end{proof}

\subsection{The space of geometric higher order modular forms}

A higher order automorphic form is said to be of {\em geometric origin} if it lies in the image of the map 
$$\Psi:H^0(B_s(X), x_0) \longrightarrow M_0^{s+1}(\Gamma)$$
The space of such {\em geometric higher order modular forms} will be denoted by $M_{Geom,0}^{s+1}(\Gamma, x_0)$.

We have an inclusion map 
$$i: H^0(B_s(\X), x_0) \longrightarrow H^0(B_s(X), x_0).$$
A geometric higher order modular form is said to be {\em cuspidal} if it  lies in the image of 
$$\Psi \circ i:H^0(B_s(\X), x_0) \longrightarrow M_0^{s+1}(\Gamma).$$
We denote  the space of {\em geometric cuspidal forms} by 
$S_{geom,0}^{s+1}(\Gamma, x_0)$.

If $K$ is in $H^0(B_s(\X),x_0)$ then for all $\sigma$ in $\Gamma_{\ca}$, where $\Gamma_{\ca}$ denotes the stabilizer of the cusp $\ca$,
$$\int_{x_0}^{\sigma x_0} K=a$$ 
where $a$ is a constant. This is because  $K=I+a$, where $I$ is in $H^0(\bar{B}(\X),x_0)$ and $a$ is a constant,  and 
$$\int_{x_0}^{\sigma x_0} I=0$$
as  the loop $\{x_0,\sigma x_0\}$ on $X$ is homotopic to the constant path on $\X$.   

A consequence of this is the following lemma, which we will have occasion to use later. 

\begin{lem} For $K$  in $H^0(B_s(\bar{X}),x_0)$ we have 
$$\int_{z_0}^{\sigma z}K -\int_{z_0}^{z}K=0$$
for all parabolic $\sigma \in \Gamma_{\ca}$ for all cusps $\ca$. In particular, for  a geometric 
higher order cusp form $f$  we have 
$$f|_{\sigma-1}\equiv 0$$
for all parabolic $\sigma \in \Gamma_{\ca}$ for all cusps $\ca$.
\label{cuspforms}
\end{lem}

\begin{proof}  Any $K$ in $H^0(B_s(\bar{X}),x_0)$ is of the form $I+a$, where $a$ is a constant and $I$ is in $H^0(\bar{B}(\bar{X}),x_0)$. Hence 
$$\int_{z_0}^{\sigma z}K -\int_{z_0}^{z}K=(a+\int_{z_0}^{\sigma z}I ) - (a+ \int_{z_0}^{z} I)$$
$$=\int_{z_0}^{\sigma z}I - \int_{z_0}^{z} I$$
From Proposition \ref{alphabeta} we have
$$\int_{z_0}^{\sigma z} I=\int_{z_0}^{z} I +  \int_{z}^{\sigma z} I + \sum_{r=1}^{s-1} \int_{z_0}^{z} I_r' \int_{z}^{\sigma z} I_r''$$
Where $I'_r$ and $I''_r$ are certain iterated integrals which are homotopy functionals but of order between $1$ and $s-1$. Since this calculation is happening with respect to forms on $\bar{X}$, $I$, $I_r'$ and $I_r''$ are all  in $H^0(\bar{B}_s(\bar{X}), \pi(z))$. Explicitly, this expansion comes from a repeated application of Propostion \ref{alphabeta} --  so the forms which constitute the $I_r'$s and $I_r''$s  are the forms which constitute $I$ and $I$ is made up of forms on $\X$. Hence all the iterated integrals are in $\bar{B}_s(\X)$. Since we also know they are homotopy invariant, they are in $H^0(\bar{B}_s(\X),\pi(z))$. 

 Since $\sigma \in \Gamma_{\ca}$  one has 
 $$\int_{z}^{\sigma z} I=0 \text{ and } \int_{z}^{\sigma z} I_r''=0$$ 
therefore most of the terms vanish and 
$$\int_{z_0}^{\sigma z} I=\int_{z_0}^{z} I.$$
A geometric cusp form is of the form $f(z)=\int_{z_0}^{z} K$ for some $K \in H^0(B_s(\bar{X}),x_0)$ for some $s$. Hence for all $z$,  
$$f|_{\sigma-1}(z)= \int_{z_0}^{\sigma z}K - \int_{z_0}^{z} K=0$$
\end{proof}

From the work of Chen\cite{chen} there is an isomorphism
$$H^0(B_s(\A^{\bullet}(X)), x_0) \longrightarrow \Hom(\Q[\pi_1(X, x_0)]/J^{s+1},\C)$$ 
and 
$$H^0(\bar{B_s}(\A^{\bullet}(X)), x_0) \longrightarrow \Hom(J/J^{s+1},\C)$$
where $\A^{\bullet}$ is any complex quasi-isomorphic to the de Rham complex $E^{\bullet}$. 
Hence we can relate these spaces of modular forms to quotients of the
group ring of the fundamental group of $X$. This motivates the
phrase `geometric origin'. Special cases of such forms were
considered in \cite{disr}.

The second graded piece of $H^0(\bar{B_s}(X),x_0)$ is isomorphic to the first cohomology group of the
modular curve $X$  --
$$ \Hom(J/J^2,\C) \simeq H^1(X,\C)$$
and this corresponds to the fact that the space of classical modular forms 
of weight $2$ is isomorphic  to the space of second order modular forms 
of weight $0$ via the map $f \rightarrow F(z)=\int_{z_0}^z f(t)dt$. Similarly the classical 
cusp forms of weight $2$ correspond to cusp forms weight $0$ and exact order $2$ via the same map -- though in that case $J$ is augmentation ideal of $\pi_1(\bar{X},x_0)$. 

More generally one can consider the completion of the group ring
with respect to the augmentation ideal 
$$\hat{\Q}[\pi(X, x_0)]=\varprojlim_s \Q[\pi(X,x_0)]/J^s.$$
This is called the Malcev completion of the group ring. The space of all modular forms of
weight $0$ and geometric origin can be interpreted as the dual of
this space.

\subsection{Hodge Structures}

Hain ( and independently, Morgan ) showed that the quotients of the
group ring with respect to powers of the augmentation ideal $J$ have
a mixed Hodge structure. 

\begin{prop}[Hain and Morgan]\cite{hain} If $X$ is an algebraic variety over $\C$ and $x_0$ is a point on $X$, there is a mixed Hodge structure on the space 
$$H^0(B_s(X),x_0)=\Hom(\Q[\pi_1(X,x_0)]/J^{s+1},\C)$$
which is natural with respect to morphisms of pointed varieties. Further, if $X$ is smooth and projective the length and weight filtrations coincide. 

\end{prop}

In particular, this holds for the algebraic curves $X=\Gamma \backslash \HH$ and their compactifications $\bar{X}$. Hence the spaces of geometric higher order modular forms  $M_{Geom,0}^{s+1}(\Gamma, x_0)$ and cusp forms $S_{Geom,0}^{s+1}(\Gamma, x_0)$ also inherit mixed Hodge structures.

A rough outline of how the the Hodge structure is obtained is as follows, at least when $X$ is a smooth curve as in our case. Essentially the same procedure works for a smooth quasi-projective variety. We follow Hain \cite{hain}. Let $\bar{X}$ denote its smooth compactification as above and $D=\bar{X}-X$. Let $E^{\bullet}(X \log D)$ denote the log-complex of smooth forms with log singularities. This complex is quasi-isomorphic to the de Rham complex $E^{\bullet}_{\C}(X)$, hence the result of Chen's \cite{chen} implies that all homotopy functionals can be obtained by iterated integrals of such forms. So it suffices to use such forms to define the Hodge structure. 

One first defines the Hodge structure on the log-complex as follows 
$$F^{p} E^{\bullet} ( X \log D ) =\{ \text{ forms with} \geq p\; dz's\}$$
$$W_{l}E^{\bullet}( X \log D )=\{ \text{ forms with}\leq l\; \frac{dz}{z}'s \}$$
and Deligne \cite{deli} showed that this induces a  Hodge structure on the cohomology of $X$ by defining the Hodge and weight filtrations to be the image of the cohomology of these filtrations in the cohomology of $X$. 

Define the filtrations on $B_s(E^{\bullet}(X \log D))$ as follows 
$$F^p B_s( E^{\bullet}(X \log D))=\text{ Span of } \left\{ \int \w_1\dots \w_r \mid  \w_i \in F^{p_i} \text{ and } \sum_{i=1}^{r} p_i \geq p \right\}$$ 
$$ W_l B_s( E^{\bullet} (X \log D)=\text{ Span of }\left\{ \int \w_1\dots \w_r \mid \w_i \in W_{l_i} \text{ and }  r+\sum_{i=1}^{r} l_i \leq  l \right \}$$ 

The Hodge and weight filtrations on $B_s(E^{\bullet}( X \log D))$ induce filtrations on $H^0( B_s(E^{\bullet}( X \log D)))$ and these define  a mixed Hodge structure. Using the map $\Psi$ we get an mixed Hodge structure on the space of geometric higher order modular forms of weight $0$. 

If $X=\X$ is a smooth projective curve then the weight filtration on $E^{1}(\X)$ is given by 
$$0=W_{-1} \subset W_{0}=E^{1}(\X)$$
hence 
$$W_l B_{s}(E^{\bullet}(\X)) =\begin{cases} B_{l}(E^{\bullet}(\X)) \text{ if } l \leq s.\\ B_s(E^{\bullet}(\X)) \text{ if } l \geq s. \end{cases}$$
Hence one has 
$$Gr_s^{W_{\bullet}} H^0(B_s(E^{\bullet}(\X), x_0)=H^0(B_s(E^{\bullet}(\X)),x_0)/H^0(B_{s-1}(E^{\bullet}(\X)),x_0)$$
so the filtration by length coincides with the weight filtration and the length graded pieces have a pure Hodge structure. 
Hence the space of geometric cusp forms of exact order $s$ has a pure Hodge structure. 

In general that is not true, as for example, if $X$ is not compact and has more than one cusp, 
$$H^0(B_1(X), x_0)/H^0(B_0(X),x_0) \simeq H^1(X,\C)$$ 
does not have a pure Hodge structure as the integral of an Eisenstein series lies in the weight
$2$ graded part of the weight filtration. So one sees that the space of higher  order modular forms of weight $0$ and order exactly $2$ does not have a pure Hodge structure. 

The Hodge structure generalizes the classical Eichler-Shimura Hodge
structure on the space of classical modular forms of weight $2$ as
that can be identified with $H^0(\bar{B}_1(X), x_0)$. In this case the
Hodge structure does not depend on the choice of $x_0$, but in
general it does.

\begin{rem}
More generally one can construct the motive underlying this Hodge
structure as the motive underlying the Hodge structure on the fundamental group is understood. This is described in the paper of  Deligne and Goncharov \cite{dego} Section 3, ( Proposition 3.4 ).  

There they show that the  the Hodge structure on the graded pieces of the group ring of the fundamental group can be realized  as the Hodge structure on the  relative cohomology groups of pairs $(X^s, \cup_{i=0}^{s} X_i )$, 
where 
\begin{itemize}

\item $X^s=X \times \dots \times X$ $s$-times
\item $X_0$ is the sub-variety given by $t_1=x_0$ -- namely $x_0  \times X^{s-1}$
\item $X_i$ is the sub-variety given by $t_i=t_{i+1}$ for $0 < i <s$ -- namely $X^{i-1}\times \Delta \times X^{s-(i+1)}$, where $\Delta$ is the diagonal in $X \times X$ in the $i^{th}$ and $(i+1)^{st}$ places.  
\item $X_s$ is given by $t_s=x_0$ -- namely $X^{s-1} \times x_0$. 

\end{itemize}
We have 
$$H^s(X^s /\cup_{i=0}^{s} X_i,\C) \simeq \Hom(J/J^{s+1},\C)=H^0(\bar{B}_s(X),x_0)$$
so precisely we get the space of weight $0$ geometric modular forms of order $s$ modulo those of order $0$. 
For example, when $s=1$ we have 
$$H^1(X/ \{x_0\},\C) \simeq H^1(X,\C) \simeq \Hom(J/J^2,\C).$$
Hence the motive underlying the Hodge structure on the space of geometric higher order modular forms of weight $0$  and order $s$  and base point $x_0$ is the motive associated to the pair $(X^s,\cup_{i=0}^{s} X_i)$. Namely, to this object one can associate a de Rham, \'{e}tale and Betti realization which are isomorphic when the field of coefficients is large enough. 

\end{rem}
\begin{rem} Classically one way of understanding Hecke operators is as follows. The space of classical modular forms of weight $k$ along with the complex conjugates of the cusp forms is, via the Eichler Shimura map, identified with the cohomology of a local system on the modular curve, and this imposes a Hodge structure on this space. Hecke operators can then be understood as the morphisms of this Hodge structure induced by certain algebraic correspondences called Hecke correspondences. 

One might hope that the same algebraic correspondences would induce morphisms on the Hodge structure of the space of geometric higher order modular forms of weight $0$ thus suggesting a way to define Hecke operators on these forms. However, unfortunately they do not, as, for example, they do not preserve base points and the Hodge structure does depend on the base point for $s>2$. Hence one cannot get  notions of Hecke eigenfunctions or Hecke eigenspaces using Hecke correspondences  and, as things stand, one cannot use this approach to define a motive of a higher order modular form. 
\end{rem}

\subsubsection{Product structure}

Let
$$\hat{M}_{Geom,0}( X, x_0)=\varinjlim_{s} M_{Geom,0}^{s}(X, x_0).$$
This space has a product structure induced by the product structure
of iterated integrals
$$B_{s_1}(X) \otimes B_{s_2}(X) \longrightarrow B_{s_1+s_2-1}(X) $$
Explicitly, this is given by the shuffle product \cite{hain} Lemma 2.11. 

For example, for two closed $1$-forms $\w_1$ and $\w_2$ with
$F_{\w_i}(z)=\int_{z_0}^z \w_i$ we have
$$F_{\w_1}(z)F_{\w_2}(z)=F_{\w_1\w_2}(z)+F_{\w_2\w_1}(z)$$

\section{Higher Weights}

We now consider the problem of constructing higher weight higher
order forms. Let $X$ and $x_0$ be as above. We have the following inductive definition for cuspidal
higher order forms. A higher order form $f$ is said to be cuspidal
if
\begin{itemize}
\item $f|_{\gamma-1}$ is cuspidal
\item $f|_{\phi-1} \equiv 0$ for all parabolic elements $\phi$ of $\Gamma$, that is 
$$\phi \in Ker\{\pi_1(X, x_0)\longrightarrow \pi_1(\bar{X}, x_0)\}.$$
\item $f$ satisfies a cuspidal growth condition -- for all cusps $\ca$ one has 
$$f|_k(\sigma_{\ca})(z) \ll e^{-cy} \text{ as } y \rightarrow \infty
 \text { uniformly in x for some constant } c>0 $$
where  $\sigma_\ca$ be an element of $SL_2(\ZZ)$ such that
$$\sigma_\ca(\infty)=\ca$$

\end{itemize}

A standard way of constructing classical modular forms is by
Poincar\'{e} series. They are defined as follows. Let $\ca$ be a
cusp, and $\sigma_\ca$ as above  so the stabilizer of $\ca$ is
$\Gamma_\ca=\sigma_{\ca}\Gamma_{\infty}\sigma_{\ca}^{-1}.$
Let $m > 0$ be an integer. For $k>2$, the Poincar\'{e} series
$P_{m,\ca}(z)=P_{m,\ca,k}(z)$  of weight $k$ is defined as follows
\begin{equation}
P_{m,\ca}(z)=\sum_{\gamma \in \Gamma_{\ca}\backslash \Gamma}
\frac{e(m\sigma_{\ca}^{-1}\gamma z)}{j(\sigma_{\ca}^{-1}\gamma,z)^k}
\end{equation}
where $e(z)=\exp(2\pi i z)$. This is a cusp form. We can also define
this when $m=0$ where this then gives the Eisenstein series of
weight $k$ corresponding to the cusp $\ca$. The Poincar\'{e} series
and Eisenstein series span the space of modular  forms for $\Gamma$
as one varies $m$ and the cusps and in fact, for a fixed cusp, one
can get a basis for the cusp forms by varying $m$ \cite{sarn}. 

We have the following generalization of Poincar\'{e} series, called
twisted Poincar\'{e} series, which give rise to higher order forms.
For order $2$ this is due to  Goldfeld \cite{gold} and O'Sullivan
\cite{goos}. In what follows we suppress the weight $k$ in the
notation.

\begin{prop} Let $k>2$ be an even integer and let $\I$ be an element
of $H^0(B_s(\X), x_0)$ and $z_0$ a point on $\HH$ lying in the fibre over $x_0$.
Then, for every cusp $\ca$ and non-negative
integer $m$  we get three twisted Poincar\'{e} series. All of these
are weight $k$, order $s+1$ modular forms and when $m>0$ they are
cusp forms.

\begin{itemize}

\item $\displaystyle{P^1_{m,\ca}(z,\I)=\sum_{\gamma \in \Gamma_{\ca}\backslash
\Gamma} \left ( \int_{z_0}^{z} \I \right
)\frac{e(m\sigma_{\ca}^{-1}\gamma
z)}{j(\sigma_{\ca}^{-1}\gamma,z)^k}}$

\item $\displaystyle {P^2_{m,\ca}(z,\I)=\sum_{\gamma \in \Gamma_{\ca}\backslash
\Gamma}\left ( \int_{z_0}^{\gamma z_0} \I \right )
\frac{e(m\sigma_{\ca}^{-1} \gamma
z)}{j(\sigma_{\ca}^{-1}\gamma,z)^k}}$

\item $\displaystyle{P^3_{m,\ca}(z,\I)=\sum_{\gamma \in \Gamma_{\ca}\backslash
\Gamma}\left ( \int_{z}^{\gamma z} \I \right )
\frac{e(m\sigma_{\ca}^{-1}\gamma
z)}{j(\sigma_{\ca}^{-1}\gamma,z)^k}}$

\end{itemize}

Further, there are relations between these modular forms coming from
the fact that the integral  of $\I$ over the path $\{z_0,\gamma z_0\}$ can be expressed, using Proposition \ref{alphabeta}, as a sum of products of integrals over  $\{z_0,z\}$,
$\{z,\gamma z\}$  and $\{\gamma z,\gamma z_0\}$, where $\{z_0,z_1\}$ denotes a path from $z_0$ to $z_1$ in $\HH$. 
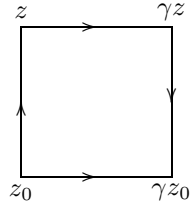
\[
\xy
(0,0)*{}; (0,20)*{}**\dir{-}
?(.5)*\dir{>};
(0,20)*{};(20,20)*{}**\dir{-}
?(.5)*\dir{>};
(20,20)*{};(20,0)*{}**\dir{-}
?(.5)*\dir{>};
(20,0)*{};(0,0)*{}**\dir{-}
?(.5)*\dir{<};
(20,-2)*{\gamma z_0};
(0,-2)*{z_0};
(0,22)*{z};
(20,22)*{\gamma z};
%
 %(0,0)*{};(0,20)*{}; **\dir{-}
%?(.47)*\dir{<}+(3,1)*{}
\endxy
\]

\end{prop}
\begin{proof}

We first show that the summand is well defined. For $P^1$ there is no question. The arguments for  $P^2$ and $P^3$ are similar and we have the following argument which we give for $P^2$.

It suffices to show that the term $\int_{z_0}^{\gamma z_0} I$ is well defined on $\Gamma_{\ca}\backslash \Gamma$. 
Let $\sigma \in \Gamma_{\ca}$ and $\gamma \in \Gamma$. Since  $I$ is in  $H^0(B_s (\bar{X}),x_0)$ and $\sigma \in \Gamma_{\ca}$, from Lemma \ref{cuspforms}, we have 
$$\int_{z_0}^{\sigma z} I - \int_{z_0}^{z}I =0$$
In particular, if $z=\gamma z_0$ we have 
$$\int_{z_0}^{\sigma \gamma z_0} I - \int_{z_0}^{\gamma z_0} I=0$$
Hence the  summand $\int_{z_0}^{\gamma z_0} I $  is well defined on $\Gamma_{\ca} \backslash \Gamma$. 

We now have a lemma regarding convergence.

\begin{lem} The series $P^{i}(z,\I)$ converge for $k>2$.
\end{lem}
\begin{proof}
For $i=1$, the series $P^1$ is simply a product of the classical
Poincar\'{e} series and $\int_{z_0}^{z} \I$, both of which converge
for $k>2$.  

For $P^3(z,\I)$ the argument is similar to the case of $\I=f$ in
\cite{dios}.  We can assume $I$ is in $H^0(\bar{B}_s(\X),x_0)$ since an element of $H^0(B_s(\X),x_0)$ differs from such an $I$ by a constant, and in the twisted Poincar\'{e} series this amounts to adding a constant multiple of a classical Poincar\'{e} series. 

We use the estimate \cite{DKMO}, Lemma 3 -- for a
classical cusp form $f$  of weight $2$, any cusp $\ca$,
$$\int_{z_0}^{z} f(w)dw \ll |\log (Im(\sigma_{\ca} z))|+1.$$
To simplify exposition, we use the notation $y_{\ca}(z)$ for
$Im(\sigma_{\ca}(z))$. Using the above estimate repeatedly, we have
that for an iterated integral $\I$ of length $r$,
\begin{equation}
\int_{z_0}^{z} \I \ll \frac{|\log^{r}(y_{\ca}(
z))|}{r!}+\frac{|\log^{r-1} (y_{\ca}(z))|}{(r-1)!} \dots +1
\label{est1}
\end{equation}
Observe that
\begin{equation}
\frac{|\log^{r}(y_{\ca}( z))|}{r!}+\frac{|\log^{r-1}(
y_{\ca}(z))|}{(r-1)!} \dots +1 < exp(|\log( y_{\ca} (z))|) < y_{\ca}
(z)+y_{\ca} (z)^{-1} \label{est2}
\end{equation}
as $exp(|\log(x)|)=x$ or $x^{-1}$. Replacing $y_{\ca} (z)$ by
$y_{\ca}( z)^{\epsilon}$, we have, for $0<\epsilon<1$, using that
$\epsilon^{r}<\epsilon$,
\begin{equation}
\int_{z_0}^{z} \I \ll \epsilon^{-r} y_{\ca} (z)^{\epsilon}+y_{\ca}
(z)^{-\epsilon}) \label{est3}
\end{equation}
To apply this to the convergence of the twisted Poincar\'{e} series
we need an estimate for $\int_{z}^{\gamma z} \I$. We have
\begin{equation}
\int_z^{\gamma z} \I = \sum_{j=0}^{r} \int_z^{z_0}
\I_{j}\int_{z_0}^{\gamma z} \I'_{j} \label{exp1}
\end{equation}
where $\I_j$ and $\I'_j$ are iterated integrals of lengths $j$ and
$r-j$ respectively. From \eqref{est3} and \eqref{exp1} we have
\begin{equation}
|\int_z^{\gamma z} \I| \ll
\epsilon^{-r}(r+1)(y_{\ca}(z)^{\epsilon}+y_{\ca}(z)^{-\epsilon})(y_{\ca}(\gamma
z)^{\epsilon}+y_{\ca}(\gamma z)^{-\epsilon}) \label{est4}
\end{equation}
So we have
$$P^3_{\ca}(z,\I) \ll \sum_{\gamma \in \Gamma_{\ca}\backslash \Gamma}\frac{(y_{\ca}(z)^{\epsilon}+y_{\ca}(z)^{-\epsilon})(y_{\ca}(\gamma
z)^{\epsilon}+y_{\ca}(\gamma z)^{-\epsilon})}{j(\gamma,z)^{-k}}$$
We have $y_{\ca}(\gamma z)=\frac{y_{\ca}(z)}{j(\gamma,z)^{2}}$.
Replacing $j(\gamma,z)^{-k}$ by $\left (\frac{y_{\ca}(\gamma
z)}{y_{\ca}(z)}\right )^{k/2}$ in the expression above we get,
$$P^{3}_{\ca}(z,\I) \ll y_{\ca}(z)^{\epsilon-k/2}\left ( \sum_{ \gamma \in \Gamma_{\ca} \backslash \Gamma}
 y_{\ca}(\gamma z)^{k/2+\epsilon} + \sum_{ \gamma \in \Gamma_{\ca} \backslash \Gamma} y_{\ca}(\gamma z)^{k/2-\epsilon} \right )
+ $$
$$ +  y_{\ca}(z)^{-\epsilon-k/2}\left ( \sum_{\gamma \in \Gamma_{\ca}
\backslash \Gamma}
 y_{\ca}(\gamma z)^{k/2+\epsilon} + \sum_{ \gamma \in \Gamma_{\ca} \backslash \Gamma} y_{\ca}(\gamma
 z)^{k/2-\epsilon} \right )$$
The sum
$$E_{\ca}(z,s)=\sum_{\gamma \in \Gamma_{\ca}
\backslash \Gamma}
 y_{\ca}(\gamma z)^{s}$$
is the classical non-holomorphic Eisenstein series for the cusp
$\ca$ and is known to be absolutely convergent in the region
$Re(s)>1$. So as long as $k/2-\epsilon >1$, that is, $k>2$,  our
series will converge.

Using the change of basepoint formula,  $P^2(z,\I)$ can be
expressed as a finite linear combination terms of the form
$P^1(z,\I') P^3(z,\I'')$ for some iterated integrals $\I'$ and
$\I''$, so converges in the same region.

\end{proof}

It remains to show the invariance property.
\begin{lem} The $P^i(z,\I)$ are higher order modular forms of weight $k$
and order $s+1$.
\end{lem}
\begin{proof}

For $P^1$ this is immediate from the earlier proposition, as $\int_{z_0}^{z} \I$ is a
weight $0$ order $s+1$ form, while $P(z)$ is order $1$, weight $k$,
so the product is order $s+1$, weight $k$. For the second one, let
$P(z,\I)=P_{m,\ca}^2(z,\I)$. We will show that $P(z,\I)|_{\beta-1}$ can be expressed as a linear combination of terms of the form $P(z,I')$ where $I'$ is an iterated integral which is a homotopy functional, but of length strictly less than that of $I$. The theorem will then follow by induction. 

We have
$$P(\beta z,\I)j(\beta,z)^{-k}-P(z,\I)$$
$$=\sum_{\gamma \in \Gamma_{\ca}\backslash
\Gamma}\left ( \int_{z_0}^{\gamma z_0} \I \right )
\frac{e(m\sigma_{\ca}^{-1} \gamma \beta
z)}{j(\sigma_{\ca}^{-1}\gamma,\beta z)^k j(\beta,z)^k}- \sum_{\gamma
\in \Gamma_{\ca}\backslash \Gamma}\left ( \int_{z_0}^{\gamma z_0} \I
\right ) \frac{e(m\sigma_{\ca}^{-1} \gamma
z)}{j(\sigma_{\ca}^{-1}\gamma,z)^k}$$
$$=\sum_{\gamma \in \Gamma_{\ca}\backslash
\Gamma}\left ( \int_{z_0}^{\gamma z_0} \I \right )
\frac{e(m\sigma_{\ca}^{-1} \gamma \beta
z)}{j(\sigma_{\ca}^{-1}\gamma \beta ,z)^k} -\sum_{\gamma \in
\Gamma_{\ca}\backslash \Gamma}\left ( \int_{z_0}^{\gamma z_0} \I
\right ) \frac{e(m\sigma_{\ca}^{-1} \gamma
z)}{j(\sigma_{\ca}^{-1}\gamma,z)^k}$$
\begin{equation}
=\sum_{\gamma \in \Gamma_{\ca}\backslash \Gamma}\left (
\int_{z_0}^{\gamma \beta^{-1} z_0} \I \right )
\frac{e(m\sigma_{\ca}^{-1} \gamma
z)}{j(\sigma_{\ca}^{-1}\gamma,z)^k}-\sum_{\gamma \in
\Gamma_{\ca}\backslash \Gamma}\left ( \int_{z_0}^{\gamma z_0} \I
\right ) \frac{e(m\sigma_{\ca}^{-1} \gamma
z)}{j(\sigma_{\ca}^{-1}\gamma,z)^k}. \label{integr}
\end{equation}
We can once more apply  Propostion \ref{alphabeta}  with $\alpha=\{z_0, \gamma
z_0\}$ and $\beta=\{\gamma z_0, \gamma \beta^{-1}z_0\}$. Since
$\{\gamma z_0, \gamma \beta^{-1} z_0\}$ is homotopic to $\{z_0,
\beta^{-1} z_0\}$ on $\Gamma \backslash \HH$, and $\I$ is homotopy
invariant, we obtain:
\begin{equation}
\int_{z_0}^{\gamma \beta^{-1} z_0} \I- \int_{z_0}^{\gamma z_0} \I=
\int_{z_0}^{\beta^{-1}z_0} \I+ \sum_{r=1}^{s-1} \int_{z_0}^{\gamma z_0} I_1^r
\int_{z_0}^{\beta^{-1}z_0} I_2^r
\label{step3}
\end{equation}
where $I_1^r$ and $I_2^r$ are sums of iterated line integrals  appearing in the proof of Theorem \ref{mainthm}. They are homotopy functionals by construction. 

Combining this with \eqref{integr} we have
\begin{equation}
P(\beta z,\I)j(\beta,z)^{-k}-P(z,\I)= \left( \int_{z_0}^{\beta^{-1}z_0} I \right) P(z) +
\sum_{r=1}^{s-1} \left ( \int_{z_0}^{\beta^{-1}z_0} I_2^r \right )P(z, I_1^r) 
\label{betaminus1}
\end{equation}
 This, by induction on $s$, is a higher order
modular form of weight $k$ and order $s$. Hence $P(z,\I)$ is a
higher order modular form of weight $k$ and order $s+1$.

The third type of Poincar\'{e} series is a higher order modular form
as the iterated integral can be expressed in terms of the first two
and products of lower order integrals using the change of basepoint
formula which comes out of Proposition \ref{alphabeta} and the fact
that
$$\int_{z_0}^{z} \I=\int_{\gamma z_0}^{\gamma z} \I$$
for $\gamma \in \Gamma$.

To show that $P_{m,\ca}(z,I)$ is cuspidal for $m>0$ using
\eqref{betaminus1} we have, by induction, that
$P_{m,\ca}(z,I)|_{\beta-1}$ is cuspidal. Further, if $\phi$ is a
parabolic element, $P_{m,\ca}(z,I)|_{\phi-1}=0$ as $\int_{z_0}^{\phi
z_0} I_2^r=0$ as all the $I_2^i$ lie in $H^0(B_s(\bar{X}, x_0))$. In
fact, this condition is also satisfied for $m=0$ when $s>0$.

\end{proof}

This completes the proof that the twisted Poincar\'{e} series give higher order modular forms. 
\end{proof}

For example, when $\I=f$, where $f$ is a weight $2$ cusp form,
either holomorphic or anti-holomorphic, then $P^2(z,f)=-P^3(z,f)$ as
in this case the integrals do not depend on the base point. In
general, however, these forms could be different, but they are
related.

In \cite{joos} certain higher order non holomorphic Eisenstein series are
constructed by twisting Eisenstein series by products of modular symbols. One can also 
twist non-holomorphic Eisenstein series by iterated integrals to get higher order 
non-holomorphic Eisenstein series. The ones constructed by Jorgenson and 
O'Sullivan are then special cases of this construction because the product of modular 
symbols can be expressed as a sum of iterated integrals via the 
shuffle product of iterated integrals. 

\subsection{Weight 2}

The case of weight $2$ modular forms requires a little more delicate
handling as the Poincar\'{e} series do not converge.  An approach to resolving 
this is to use the ideas of Diamantis and O'Sullivan \cite{dios}.
They overcome this problem by defining it as a function obtained as
a special value of the analytic continuation of a certain Poincare
series with an additional factor which makes it converge.

Precisely, for an integer $m$ and a cusp $\ca$, define
$$\displaystyle{\MZ_{m,\ca}(z,s,\I)=\sum_{\gamma \in
\Gamma_{\ca}\backslash \Gamma} \left ( \int_{z_0}^{\gamma z_0} \I
\right )\frac{e(m\sigma_{\ca}^{-1}\gamma
z)Im(\sigma_{\ca}^{-1}\gamma
z)^s}{j(\sigma_{\ca}^{-1}\gamma,z)^2}}$$
As a function of $s$ this has an analytic continuation to the entire
complex plane. In particular, one can put $s=0$ and the resulting
function
$$P_{m,\ca}(z,\I)=\MZ_{m,\ca}(z,0,\I)$$
is weight $2$, order $s+1$ modular form. The argument is similar
to that of Diamantis and Sim \cite{disi}. However, as the details are complicated, we
will not deal with this case in the remaining part of this  paper.

\subsection{Spaces of higher order and higher weight modular forms}

In the previous sections we  constructed some examples of higher
weight, higher order forms. We would like to define the space
$M_{geom,k}^{s+1}(\Gamma, x_0)$ to be the largest space we can get
from the constructions above. For that we first define a space of
{\em primitive} forms.

The space of {\em primitive cusp} forms $SPM_{geom,k}^{s+1}
(\Gamma, x_0)$ is defined to be the space spanned by the forms
$P_{m,\ca}(z,\I)$ and their complex conjugates over all cusps $\ca$
and all positive integers $m$. The space of {\em primitive modular}
forms $PM_{geom,k}^{s+1}(\Gamma, x_0)$ is then the space spanned by
$SPM_{geom,k}^{s+1}$ and the Eisenstein series $E_{\ca}(z,\I)=P_{0,\ca}(z,\I)$ over
all cusps $\ca$ where $\I$ is in $H^0(B_s(\X),z_0)$. As the space of
smooth modular forms of order $s$ and weight $k$  is finite
dimensional, this space is finite dimensional. In weight $0$ 
all the  forms constructed above as anti-derivatives of iterated integrals  are said to be primitive.

There is a product structure on the space of modular forms of
higher order. This was first introduced by O'Sullivan. If $F$ is a
modular form of weight $k_1$ and order $s_1$ and $G$ is a modular
form of weight $k_2$ and order $s_2$, then from \cite{dich}, the
Rankin-Cohen bracket for $N=0$, we have that $FG$ is a modular form
of weight $k_1+k_2$ and order $s_1+s_2-1$.

We can see this easily in the weight $0$ case -- The product of two
iterated integrals of orders $s_1$ and $s_2$, which correspond to
modular forms of order $s_i+1$, is an iterated integral of order
$s_1+s_2$ whose anti-derivative is an modular form of order
$s_1+s_2+1=(s_1+1)+(s_2+1)-1$.

So we finally define the space of geometric higher order modular
forms of order $s$ and weight $k$, $M_{geom,k}^s(\Gamma, x_0)$  to be the
algebra generated by the primitive forms and similarly the space of
geometric higher order cusp forms $SM_{geom,k}^s(\Gamma, x_0)$ to be the
subalgebra generated by the primitive cusp forms. A weight $k$ form, therefore, is a sum 
of products of lower weight primitive forms. To study this space in more detail  we have to 
include  the space of weight $2$ higher order forms which, as mentioned above,  requires more delicate 
handling, so in what follows we will only consider  the primitive spaces.

\subsection {Hodge Structures}

We can define an {\em ad hoc} Hodge structure on the spaces of primitive modular
forms, $PM_{geom,k}^s$. The weight and Hodge  filtration are defined
as follows. Recall that on $H^0(B_s(\X), x_0)$, the (Hodge) weight
filtration and filtration by length coincide. Define
\begin{itemize}
\item $W_{l} PM_{geom,k}^{s+1}=\text{ Span of  }\left\{ P_{m,\ca,k}(z, I), \bar{P}_{m,\ca,k}(z, I) \text { and } E_{\ca,k}(z, K) \right\}$ such that $I \in W_{l-(k-1)} H^0(B_s(\X), x_0),  m >0$ and $K \in W_{l-k} H^0(B_s(\X), x_0)$ and 
where  all the Poincare and Eisenstein series are of weight $k$.

\item $F^p PM_{geom,k^{s+1}}=\text{ Span of } \left\{ P_{m,\ca,k}(z,I),  \bar{P}_{m,\ca,k}(z, J) \text{ and } E_{\ca,k}(z,K)  \right \}$ such that $m >0$,  $I \in F^{p-(k-1)} H^0(B_s(\X), x_0)$,   $J \in F^{p} H^0(B_s(\X), x_0)$ 
and   $K \in F^{p-(k/2)} H^0(B_s(\X), x_0)$.
\end{itemize}
So, for example, if $k=4$, $s=2$, the weight filtration on
$PM_{geom,4}^{3}$ is as follows
\begin{itemize}
\item $W_{0}=W_1=W_2=0$.
\item $W_3=$ Span of $P_{m,\ca}(z)$ and $\bar{P}_{m,\ca}(z)$ =
holomorphic and anti-holomorphic cusp forms of weight $4$.
\item $W_4=$ Span of Eisenstein series $E_{\ca}(z)$ of weight $4$ and span of
$P_{m,\ca}(z,I)$ and $\bar{P}_{m,\ca}(z,I)$, where $I$ is in
$H^0(B_1(\X),z_0)$.
\item $W_5=$ Span of $W_4$,$P_{m,\ca}(z,I),\bar{P}_{m,\ca}(z,I)$ where $I$
is in $H^0(B_2(\X), x_0)$ and $E_{\ca}(z,J)$ where $J$ is in
$H^0(B_1(\X),z_0)$.
\item $W_6=PM_{geom,4}^{3}=W_{i}, i\geq 6=$ Span of $W_5$ and $E_{\ca}(z,J)$ where $J$ is in
$H^0(B_2(\X), x_0$.
\end{itemize}
and the Hodge filtration is given as follows.
\begin{itemize}
\item $F^0=PM_{geom,4}^{3}$
\item $F^1=$ Span of $\bar{P}_{\m,\ca}(z,I)$, where $I \in F^1
H^0(B_2(\X), x_0)$, $P_{\m,\ca}(z,J)$ where $J \in H^0(B_2(\X), x_0)$
and $E_{\ca}(z,K)$ where $K \in H^0(B_2(\X), x_0)$
\item $F^2=$ Span of $\bar{P}_{\m,\ca}(z,I)$, where $I \in  F^2
H^0(B_2(\X), x_0)$,$P_{\m,\ca}(z,J)$ where $J \in H^0(B_2(\X), x_0)$
and $E_{\ca}(z,K)$ where $K \in H^0(B_2(\X), x_0)$.
\item $F^3=$ Span of $P_{\m,\ca}(z,J)$ where $J \in H^0(B_2(\X), x_0)$
and $E_{\ca}(z,K)$ where $K \in F^1 H^0(B_2(\X), x_0)$.
\item $F^4=$ Span of $P_{\m,\ca}(z,J)$ where $J \in F^1
H^0(B_2(\X), x_0)$ and $E_{\ca}(z,K)$ where $K \in F^2
H^0(B_2(\X), x_0)$.
\item $F^5=$ Span of $P_{\m,\ca}(z,J)$ where $J \in F^2
H^0(B_2(\X), x_0)$.
\item $F_i=0\;\; i \geq 6$.
\end{itemize}

We expect that this Hodge structure can be used to define a Hodge
structure on the full space of geometric higher order forms using
the fact that it is the algebra generated by these forms, but since
that requires the weight $2$ case as well, we will not consider it
here.

At the moment it is not clear whether there is some natural geometric structure underlying this Hodge structure, and so one cannot say anything about naturality or functoriality. However, there is some recent work of Anton Dietmar \cite{diet} relating higher order forms with Lie algebra cohomology and one might hope  that this Hodge structure is related to a natural Hodge structure on those cohomology groups.

\bibliographystyle{alpha}

\bibliography{HOMFRef}

\noindent Ramesh Sreekantan\\ 
TIFR Centre for Applicable Mathematics\\
Post Bag No. 6503\\
Sharada Nagara, Chikkabommasandra\\ 
Bangalore, 560 065, Karnataka, India \\
Email: ramesh\@@math.tifr.res.in

\end{document}